\documentclass[11pt]{amsart}
\usepackage{amsmath,amssymb,mathrsfs, xcolor, geometry, graphicx, subcaption, float,enumerate}
\usepackage{hyperref}

\newtheorem{theorem}{Theorem}
\newtheorem{proposition}[theorem]{Proposition}
\newtheorem{corollary}[theorem]{Corollary}
\newtheorem{lemma}[theorem]{Lemma}
\newtheorem{conjecture}[theorem]{Conjecture}

\theoremstyle{definition}

\newtheorem{remark}[theorem]{Remark}

\makeatletter
\newcommand*\owedge{\mathpalette\@owedge\relax}
\newcommand*\@owedge[1]{%
  \mathbin{%
    \ooalign{%
      $#1\m@th\bigcirc$\cr
      \hidewidth$#1\m@th\wedge$\hidewidth\cr
    }%
  }%
}
\makeatother

\makeatletter 
\def\l@subsection{\@tocline{2}{0pt}{1pc}{5pc}{}} \def\l@subsection{\@tocline{2}{0pt}{2pc}{6pc}{}} \makeatother

\newcommand{\Ric}{\mathrm{Ric}}
\newcommand{\scal}{\mathrm{scal}}
\newcommand{\Hess}{\mathrm{Hess}}

\begin{document}

\title{Linear stability and instability of Kähler Ricci solitons}
\author{Keaton Naff \quad  \quad Tristan Ozuch}

\maketitle

\begin{abstract}
    We show that the recently discovered BCCD shrinking soliton is linearly unstable, by extending the approach of \cite{chi04} and \cite{hm11}, via recent work the \cite{cm21} on gradient shrinking Ricci solitons. On the other hand, we prove that the weighted $L^2$-spectra of the weighted Lichnerowicz Laplacians of steady and expanding Kähler Ricci solitons are nonpositive in real dimension $4$. We additionally determine the linear stability of the orbifold singularities of Kähler solitons: shrinkers are unstable, steadies are neutrally stable and expanders are strictly stable. All of these results follow from new Weitzenböck formulae for the weighted Lichnerowicz Laplacian specialized to Kähler metrics.
\end{abstract}

\section*{Introduction}

There is a growing effort to develop a theory of \emph{generic Ricci flow} in dimension four and higher. The hope is that for an open dense set of initial metrics, the flow should only develop \emph{linearly stable} singularity models and undergoes a manageable list of topological surgeries. The motivation for such a theory is twofold. First, as in dynamical systems, stability yields a robust description of typical evolutions. This was an important aspect of the construction of a 3D \textit{Ricci flow through singularities}, \cite{kl17, bk22}. Second, reducing the myriad of possible singularity models to a short list of \textit{stable} shrinking and steady solitons makes the geometric and topological surgeries performed canonical. Indeed, in principle, unstable singularity models can be avoided by perturbation, and thus only stable singularity models should remain. Subsequently, stable expanding (and steady) solitons should provide the models needed to restart the flow. Remarkably, this heuristic picture has been achieved for embedded, codimension one mean curvature flow in dimension two; see \cite{cm12,bk23, ccms24a, ccs23,ccms24b} and the many references therein. Motivated by Perelman's groundbreaking work in dimension three \cite{p02,p03}, a long-term goal is to obtain a \textit{generic canonical neighborhood theorem}, modeling the geometry of high-curvature regions by \textit{linearly stable} shrinking and steady solitons, and classifying the associated topological changes via stable steady and expanding solitons. 

The motivations above originate, in part, from foundational ideas of Richard S. Hamilton. Indeed, Hamilton introduced Ricci shrinking solitons as models of singularity formation early on in his development of Ricci flow. Shortly after Perelman introduced his remarkable monotonic entropy functional to the Ricci flow, Cao, Hamilton, and Ilmanen \cite{chi04} observed that the functional may be used to investigate the stability and instability of singularity formation. In this direction, classifying stable solitons is an interesting and valuable open problem. \\

\noindent\textbf{Question 1} (Classification of compact stable shrinking solitons)\textbf{.} \textit{Is the round metric on $\mathbb{S}^4$ the only \emph{smooth, simply-connected, compact, strictly linearly stable} shrinking Ricci soliton? Is it the only compact stable \emph{orbifold} shrinking Ricci soliton in the sense of \cite{doorb}?}
\\

\noindent\textbf{Question 2} (Classification of noncompact stable shrinking solitons)\textbf{.} \textit{Are the only smooth, simply-connected, stable shrinking Ricci solitons either the round sphere $\mathbb{S}^4$, the round cylinders $\mathbb{S}^3 \times \mathbb{R}$, $\mathbb{S}^2 \times \mathbb{R}^2$, or the \emph{blowdown} soliton of \cite{fik}? What are the stable \emph{orbifold} examples?} \\

The smooth setting of Question 1 is attributed to Hamilton  -- see \cite{cao06, cao10} as well as  \cite{cz12, cz24}. 

In this paper, we contribute to this program by establishing the linear instability of the shrinking soliton found in \cite{bccd}. This resolves the stability question for the last \textit{currently known} shrinking soliton whose stability was previously undetermined, and the last \textit{K\"ahler} shrinking Ricci soliton, which have now been classified. We also prove the stability of the large class of K\"ahler steady and expanding solitons in the sense of the sign of the weighted-$L^2$-spectrum of their weighted Lichnerowicz Laplacian. We additionally determine the stability of K\"ahler (and non-K\"ahler shrinking) orbifold solitons in the sense of \cite{doorb}.

\subsection*{Instability of the BCCD soliton}

Outside of the generalized cylinders $\mathbb{R}^4, \mathbb{S}^4, \mathbb{S}^3 \times \mathbb{R}, \mathbb{S}^2 \times \mathbb{R}^2$, and the Blowdown soliton of \cite{fik} whose stability is proven in \cite{no25}, all other currently known smooth, simply-connected gradient Ricci shrinking solitons have been shown to be unstable. 
The stability of the recent soliton of \cite{bccd}, however, had remained unknown until now. See Appendix \ref{app:solitons}, Table \ref{tab:page13}.

Let $L^2_f :=  L^2(e^{-f}d\mu_g)$. We say $h\in H^2_f$ if $h, \nabla h, \nabla^2 h\in L^2_f$. 

\begin{theorem}[Instability of the BCCD soliton]\label{thm: instab BCCD}
    The soliton found in \cite{bccd} is linearly unstable in the sense of \cite{chi04}. That is:
    \begin{enumerate}
        \item There exists a $2$-tensor $h\in H^2_f$ such that  $h\perp_{L^2_f} \Ric(g)$, $\operatorname{div}_fh=0$ and
        \[
        \langle L_f h, h\rangle_{L^2_f} >0.
        \]
        \item There exists a non-vanishing ancient solution of the equation $\partial_th = L_f h$ with $\|h\|_{L^2_f}$ uniformly bounded in time.
    \end{enumerate}
\end{theorem}
\begin{remark}
    This answers the question of the stability of all of the \emph{Kähler} shrinking solitons classified in \cite{bccd,lw23,cds24}. 
\end{remark}

\subsection*{Stability of steady and expanding Kähler solitons}

The dynamical stability \cite{der15} and dynamical instability \cite{doorb} of gradient expanding and steady solitons are respectively determined by the absence and presence of positive numbers in the $L^2_f$-spectrum of the Lichnerowicz Laplacian. We show that, in this sense, all expanding and steady Kähler-Ricci solitons are \textit{stable} in real dimension $4$.

\begin{theorem}[Stability of steady and expanding Kähler–Ricci solitons]\label{thm: stab}
Let $(M^4,g,J,f)$ be a complete gradient Kähler–Ricci soliton of real dimension $4$ solving
$$\Ric_g+\Hess f=\frac{\lambda}{2}g,$$
and let $L_f$ denote the weighted Lichnerowicz Laplacian acting on symmetric $2$-tensors, self-adjoint on $L^2_f$. Then, on a steady or expanding soliton, i.e. $\lambda\leqslant 0$, for all $h\in H^2_f$, one has
\[
\langle L_f h,h\rangle_{L^2_f}\leqslant 0.
\]
\end{theorem}

\begin{remark}
    This shows in particular that the Kähler expander used to restart the flow from the FIK blowdown singularity in \cite{fik} is $L^2_f$-stable. This supports the idea that the associated surgery proposed by \cite{gs} is stable. Since strict stability implies local uniqueness, this is a first step toward establishing uniqueness of these expanding solitons.
\end{remark}

\subsection*{Stability and instability of orbifold singularities of Kähler solitons}

The recent compactness theory of \cite{bam1,bam2,bam3} shows that finite-time singularities of a 4-dimensional Ricci flow  can, in general, be traced back to an \textit{orbifold} Ricci shrinking soliton, rather than a smooth Ricci shrinking soliton. Singular tangent solitons have been observed in \cite{doorb}, which introduced a notion of stability of orbifold singularities with respect to their expected origin: bubbling off of known Ricci-flat ALE metrics. This notion of stability depends on the spectrum of a weighted selfdual curvature denoted $\overline{\mathbf{R}}{}^+$ at the orbifold points: positive eigenvalues indicate instability.

\begin{theorem}\label{thm: stab orb}
    Let $(M^4,g,f)$ be a Kähler Ricci soliton \textit{orbifold} singular at $p_o$ with singularity in $SU(2)$ and solving 
$$\Ric_g+\Hess f=\frac{\lambda}{2}g.$$
Then, the eigenvalues of the weighted selfdual curvature of \cite{doorb} are (up to a positive multiplicative constant depending upon conventions):  
$$ \operatorname{spec}\overline{\mathbf{R}}{}^+ _{p_0}= \left(\lambda,\lambda-\frac{\operatorname{scal}_{p_0}}{2},\lambda-\frac{\operatorname{scal}_{p_0}}{2}\right), $$
with the Kähler form being the first eigendirection. In particular, it is \emph{orbifold point stable} in the sense of \cite{doorb} when $\lambda<0$, \textit{semistable} when $\lambda=0$, and \emph{unstable} when $\lambda>0$.  

\end{theorem}

\begin{remark}
    In particular, we detect that steady solitons with $\lambda=0$ have exactly \emph{one} neutral direction in the direction of their Kähler form at their orbifold singularities. This is consistent with the existence of Kähler steady soliton \textit{desingularizations} in \cite{bm24}. 
\end{remark}

\begin{conjecture}\label{conj: ancient KRF orb}
    Let $(M^4,g,f)$ be a Kähler Ricci shrinking soliton \textit{orbifold} singular at $p_o$ with singularity group in $U(2)$. Then, there exists a smooth ancient \emph{Kähler}-Ricci flow whose tangent soliton at $-\infty$ is $(M^4,g,f)$.
\end{conjecture}

\subsection*{Acknowledgements}

The authors thank Alix Deruelle and Ronan Conlon for pointing out the geometric description of \cite{ccd24} of the soliton of \cite{bccd}, which let us compute its central density in Table \ref{tab:page13}. During this project, TO was partially supported by the National Science Foundation under grant DMS-2405328 and KN was partially supported by the National Science Foundation under grant DMS-2103265. The authors thank Louis Yudowitz for pointing out to us that the asymptotic cone of the blowdown soliton had positive curvature operator, and that the results of \cite{gs} could be applied to manifolds with associated conical singularities.

The first author is deeply grateful to have learned Ricci flow in part from Richard S. Hamilton as a graduate student at Columbia University. Richard's vision and generosity have shaped the direction of this work. 

\section{The Weitzenb\"ock formula for K\"ahler solitons}

In this section, we state our key Weitzenb\"ock formula. A version of this formula was used by Hall and Murphy \cite{hm11} to show the instability of nontrivial compact K\"ahler Ricci shrinking solitons. We state it here for general gradient Ricci solitons and will specialize it to the K\"ahler setting in the next section. 

In Appendix \ref{app:weitzenbock}, we give a detailed summary of our conventions and the derivation that leads to the formula below. Here, we summarize the main ingredients.
\begin{itemize}
    \item Sections $T$ and $U$ of $\Lambda^2 \otimes \Lambda^2$ (including curvature operators) compose via $(T \circ U)_{ijkl} = T_{ijpq} U_{pqkl}$ and act on $2$-forms $\phi$ by $T(\phi)_{ij} = T_{ijkl}\phi_{kl}$. The trace-map $\mathrm{tr} : \Lambda^2 \otimes \Lambda^2 \to \mathrm{Sym}^2(T^\ast M)$, yielding symmetric $2$-tensors, is defined by $\mathrm{tr}(T)_{ik} = \frac{1}{2}(T_{ipkp} + T_{kpip})$.  With this convention, $\mathrm{Id}_{\Lambda^2} = \frac{1}{4} g \owedge g$, where $\owedge$ denotes the Kulkarni-Nomizu product \eqref{eq:kn-formula}.
    \item The Hodge star decomposes $\Lambda^2 = \Lambda^-_g \oplus \Lambda^+_g $ into anti-self-dual and self-dual 2-forms. This also leads to a decomposition of curvature operators. In particular, the Weyl tensor has a decomposition $W = W^+ + W^-$ with $W^\pm \in \Lambda^g_{\pm} \otimes \Lambda^g_{\pm}$. 
    \item The tensor product $\Lambda_g^- \otimes \Lambda_g^+$ is in bijection with the bundle of traceless symmetric $2$-tensors $\mathrm{Sym}^2_0(T^\ast M)$ via the trace map. 
    \item On sections of $\Lambda^k \otimes \Lambda^l$, there is a natural left-acting  differential $d^L$ and corresponding co-differential $\delta^L$. These yield a left-acting Hodge Laplacian $\Delta^L_H = -(d^L \delta^L + \delta^L d^L)$, related to the rough Laplacian $\Delta$ by a Weitzenb\"ock formula.\footnote{Right-acting operators also exist, but the left-acting ones lead to a formula involving $W^+$, the component of the Weyl tensor determined by the K\"ahler condition.} 
    \item Given a function $f$, one obtains a weighted co-differential by
    \[
    \delta^L_{f, 0} := e^f \delta^L e^{-f},
    \]
    leaving $d^{L}$ unchanged. Then $\delta^L_{f, 0}$ is formally $L^2_f$-adjoint to $d^L$. This gives rise to an $L^2_f$-self-adjoint weighted Hodge Laplacian 
    \[
    \Delta^{L}_{H, f, 0} = - (d^L \delta^L_{f, 0} + \delta^L_{f, 0} d^L),
    \]
    related to $\Delta_f$ by a weighted Weitzenb\"ock formula. 
    \item Importantly, one has 
    \[
    \langle \mathrm{tr}(\Delta^L_{H, f, 0} S), \mathrm{tr}(S) \rangle_{L^2_f} = \frac{1}{4} \langle \Delta^L_{H, f, 0} S, S \rangle_{L^2_f} \leq 0.
    \]
\end{itemize}

\begin{proposition}\label{prop:main-weitzenbock}
    Let $(M, g, f, \lambda)$ be a 4-dimensional gradient Ricci soliton satisfying 
    \[
    \Ric_g + \Hess f = \frac{\lambda}{2} g.
    \]
    If $S$ is a section of $\Lambda^-_g \otimes \Lambda^+_g$ and $h = u\, g + \mathrm{tr}(S)$, then
    \begin{align}\label{eq:main-weitzenbock-formula}
    L_f h &= (\Delta_f u) g  + 2u \, \Ric + \mathrm{tr}\Big( \Delta^L_{H, f,0} S + S \circ  \big(W^++ \big(\lambda - \frac{\scal}{3} \big)\mathrm{Id}_{\Lambda^+_g}\big)\Big).
    \end{align}
\end{proposition}

Nonpositivity of the zeroth-order term, $W^ + + (\lambda - \frac{\scal}{3} )\mathrm{Id}_{\Lambda^+_g}$, yields the stability of a soliton under the given deformation. Alternatively, the vanishing of $\Delta^L_{H, f, 0} S$ with nonnegativity of zeroth-order term yields instability.

\section{Stability of gradient Kähler steady and expanding solitons}

We now specialize and apply the Weitzenböck formula of Proposition \ref{prop:main-weitzenbock} to Kähler metrics and prove that the spectrum of the weighted Lichnerowicz Laplacian of steady and expanding Kähler Ricci solitons is nonpositive. 

\subsection{Symmetric $2$-tensors and $2$-forms on Kähler surfaces and selfduality}

Let $(M^4,g,J,\omega)$ be a Kähler surface. Then the bundle of complex $(1,1)$-forms is
$$\Lambda^{(1,1)}\subset \Lambda^2\otimes\mathbb C,\qquad
\Lambda^2\otimes\mathbb C=\Lambda^{2,0}\oplus\Lambda^{(1,1)}\oplus\Lambda^{0,2}.$$
We write $\Lambda^{1,1}:=\{\alpha\in\Lambda^2:\alpha(J\cdot,J\cdot)=\alpha\}$ for the \emph{real} $(1,1)$-forms. 
Then $\omega\in\Lambda^{1,1}$ and there is an orthogonal splitting
$$\Lambda^{1,1}=\mathbb R\,\omega\ \oplus\ \Lambda^{1,1}_0,
\qquad
\Lambda^{1,1}_0:=\{\beta\in\Lambda^{1,1}:\beta\wedge\omega=0\}.$$
For a Kähler surface, with the complex orientation, there is a further splitting 
$$\Lambda^2=\Lambda^+_g\oplus\Lambda^-_g,$$
$$\Lambda^+_g=\langle\omega\rangle\oplus\Re\Lambda^{2,0},\qquad
\Lambda^-_g=\Lambda^{1,1}_0.$$
We will denote $\pi_{(1,1)}$ the projection onto the $(1,1)$-part, whose complement is $\Re\Lambda^{2,0}$, so that
\begin{equation}
    \pi_{(1,1)}\alpha=\tfrac12\left(\alpha+\alpha(J\cdot,J\cdot)\right).
\end{equation}

For $h$ a symmetric $2$-tensor, set $(J\cdot h)(X,Y):=h(JX,JY)$, and denote
\[
h_0:=h-\tfrac14(\operatorname{tr}_g h)\,g,\qquad
h_{0,I}:=\tfrac12\big(h_0 + J \cdot h_0\big) =: h_I-\tfrac14(\operatorname{tr}_g h)\,g,\qquad h_A=h_{0,A}:=\tfrac12\big(h_0 - J \cdot h_0\big).
\]
Then, define subbundles of $\mathrm{Sym}^2 = \mathrm{Sym}^2 (T^\ast M)$ (of ranks $4, 3$, and $6$)
\[
\operatorname{Sym}^2_{I}=\{h:\,J\cdot h=+ h\}, \qquad \operatorname{Sym}^2_{0,I}:=\operatorname{Sym}^2_{0}\cap \operatorname{Sym}^2_{I}, \qquad \operatorname{Sym}^2_{A} = \operatorname{Sym}^2_{0,A} =\{h:\,J\cdot h=- h\},
\]
so that
\[
\operatorname{Sym}^2=\mathbb{R}g\oplus\operatorname{Sym}^2_0,\qquad
\operatorname{Sym}^2_0=\operatorname{Sym}^2_{0,I}\oplus\operatorname{Sym}^2_{0,A}.
\]

On the other hand, on any oriented $4$-manifold, as used to prove our Weitzenböck formula, the trace operation $\operatorname{tr}:\Lambda^2\otimes\Lambda^2\to \operatorname{Sym}^2$ provides an isomorphism
$$\operatorname{Sym}^2_0\cong\Lambda^-_g\otimes\Lambda^+_g.$$
In terms of symmetric $2$-tensors, this gives
\[
\operatorname{Sym}^2_{0,I}\cong\Lambda^{1,1}_0\otimes\langle\omega\rangle ,\qquad
\operatorname{Sym}^2_{0,A}\cong\Lambda^{1,1}_0\otimes\Re\Lambda^{2,0} .
\]
Since the metric $g$ is proportional to $\operatorname{tr}(\omega\otimes\omega)$ (with a constant depending upon conventions), using the operator $\operatorname{tr}$ we obtain the isomorphisms:
\[
\operatorname{Sym}^2_I\cong\Lambda^{1,1}\otimes\langle\omega\rangle=\big(\langle \omega\rangle\oplus \Lambda^{1,1}_0\big)\otimes \langle \omega\rangle =\left(\langle \omega\rangle\oplus \Lambda^-_g\right)\otimes \langle \omega\rangle, \text{ and}
\]
\[
\operatorname{Sym}^2_A=\operatorname{Sym}^2_{0,A}\cong\Lambda^{1,1}_0\otimes\Re\Lambda^{2,0}.
\]

\subsection{Action of the curvature in the K\"ahler setting}

We first recall that the curvature operator, seen as an element of $\Lambda^2\otimes\Lambda^2$ of a Kähler surface is $J$-invariant, hence it lies on
\begin{equation}
    (\Lambda^2\otimes\Lambda^2)_I= \big(\langle\omega\rangle\otimes\langle\omega\rangle\big)\oplus \big(\langle\omega\rangle\otimes \Lambda^-_g\big)\oplus \big(\Lambda^-_g\otimes \langle\omega\rangle\big) \oplus\big(\Lambda^-_g\otimes \Lambda^-_g\big),
\end{equation}
In particular, the Ricci curvature is of the form $ \rho\otimes\omega  + \omega \otimes \rho \in (\Lambda^{1,1}\otimes \langle\omega\rangle )\oplus ( \langle\omega\rangle\otimes \Lambda^{1,1})$ and the selfdual curvature $\frac{\operatorname{scal}}{6}\operatorname{Id}_{\Lambda^+_g}+W^+\in \langle\omega\rangle\otimes \langle\omega\rangle$. This crucially determines the action of $\frac{\operatorname{scal}}{6}\operatorname{Id}_{\Lambda^+_g}+W^+$ on $\Lambda^+_g=\langle\omega\rangle\oplus\Re\Lambda^{2,0}$:
\begin{equation}
    \Big(W^+ +\frac{\operatorname{scal}}{6}\operatorname{Id}_{\Lambda^+_g}\Big)(\omega) = \frac{\operatorname{scal}}{2} \omega, \qquad  \Big(W^+ + \frac{\operatorname{scal}}{6}\operatorname{Id}_{\Lambda^+_g}\Big)\Big|_{\Re\Lambda^{2,0}}=0.
\end{equation}
Consequently, we have
\begin{equation}\label{eq:action curvature Kähler}
    \Big(W^+ +\Big(\lambda-\frac{\operatorname{scal}}{3}\Big)\operatorname{Id}_{\Lambda^+_g}\Big)(\omega) = \lambda \omega, \qquad\Big(W^+ +\Big(\lambda-\frac{\operatorname{scal}}{3}\Big)\operatorname{Id}_{\Lambda^+_g}\Big)\Big|_{\Re\Lambda^{2,0}}=\left(\lambda-\frac{\operatorname{scal}}{2}\right)\operatorname{Id}_{\Re\Lambda^{2,0}}.
\end{equation}

\subsection{The weighted Hodge Laplacian and divergence in the  Kähler setting}

Having established that $\Lambda^-_g \otimes \Lambda^+_g = \Lambda^{1,1}_0 \otimes (\langle \omega \rangle \oplus \mathfrak{R} \Lambda^{2, 0})$ and examined the action of $W^++ (\lambda - \frac{\scal}{3}) \mathrm{Id}_{\Lambda^+_g}$ on $\Lambda^+_g$, it remains to further examine $\Delta^L_{H, f, 0} S$. We begin with a general identity: 
\[
\Delta^L_{H, f, 0} (\gamma \otimes \phi) = (\Delta_{H, f,0} \gamma) \otimes \phi + 2 \nabla_k \gamma \otimes \nabla_k \phi + \gamma \otimes \Delta_f \phi + R \# (\gamma \otimes \phi).
\]
This follows from the definitions and formulas in Sections~\ref{sec:classical-weitzenbock}-\ref{sec:weighted-operators}. If $\phi = \omega$, the middle two terms vanish since the Kähler form is parallel. If, additionally,  $\gamma \in \Lambda^-_g$, then the last term also vanishes. Indeed, since the operator $\#$ preserves the duality decomposition of $\Lambda^2 \otimes \Lambda^2$, one has 
\[
R \# (\gamma \otimes \omega) = \frac{1}{2}(\mathring{\Ric} \owedge g) \# (\gamma \otimes \omega)
\]
because $\mathring{\Ric} \owedge g$ is the only component of the curvature in $\Lambda^-_g \otimes \Lambda^+_g$. In the K\"ahler setting $ \mathring{\Ric}\owedge g = \mathring{\rho} \otimes \omega + \omega \otimes \mathring{\rho}$ where $\mathring{\rho} = \rho - \scal \, \omega \in \Lambda^{1,1}$ and $\rho$ is the Ricci form. It follows from the defining identity \eqref{eq:sharp-lie} that $\mathring{\rho} \otimes \omega  \# \gamma \otimes \omega  = \omega \otimes \mathring{\rho}  \# \gamma \otimes \omega = 0$.

Hence, for $S_I = \gamma \otimes \omega \in\Lambda^{1,1}\otimes \langle\omega\rangle$, one has $\Delta^L_{H,f 0} S =( \Delta^L_{H, f, 0} \gamma ) \otimes \omega$, and therefore  
\begin{equation}\label{eq:Laplacian J inv Kähler}
    \operatorname{tr}\left(\Delta^L_{H, f, 0} S_I\right) = \mathrm{tr}\big(( \Delta^L_{H, f, 0} \gamma ) \otimes \omega\big) = (\Delta_{H, f, 0}\gamma)\circ  \omega.
\end{equation}
\begin{remark}
    This is the same formula as in \cite[Lemma 4.3]{hm11}. 
\end{remark}
We may now define the space of \textit{weighted} harmonic $2$-forms, denoted 
\[
\mathcal{H}_f^{1,1}:= \left\{ \gamma\in \Lambda^{1,1}\cap H^1_f :  d\gamma =  \delta_{f,0}\gamma=0 \right\}.
\]

When analyzing the linear stability or instability of gradient Ricci solitons in the direction of a deformation $h$, one may assume without loss of generality that $\mathrm{div}_f(h) = 0$ (see, for instance, the discussion in \cite[Section 2.2]{no25}). 
In the K\"ahler setting, we have the following nice formula for the weighted divergence. If $h_I = \operatorname{tr}(S_I)$ with $S_I= \gamma\otimes \omega $, then
\begin{equation}\label{eq: weighted div free}
    \operatorname{div}_fh_I=0 \iff \delta_{f,0} \gamma=0.
\end{equation}
\begin{remark}
    The Ricci form $\rho\in \Lambda^{1,1}$, always satisfies $d\rho=0$, $\delta_{f,0}\rho=0$, so it is always in the kernel of $\Delta_{H,f,0}$ and $\operatorname{div}_f$-free. 
\end{remark}

\subsection{Weitzenböck formulae on Kähler Ricci solitons}

Let $(M,g,J,f)$ be a Kähler Ricci soliton satisfying 
\[
\Ric+\Hess f=\frac{\lambda}{2}g.
\]
Let $\omega(\cdot, \cdot) = g(J\cdot, \cdot)$ denote the K\"ahler form. Let $h_I$ and $h_A$ denote the $J$-invariant and $J$-anti-invariant parts of a symmetric $2$-tensor $h$, and define $S_I\in \Lambda^{1,1}\otimes \langle\omega\rangle$ and $S_A\in \Lambda^{1,1}_0\otimes \Re\Lambda^{2,0}$ through 
\begin{equation}
    \operatorname{tr}(S_I) = h_I\qquad \text{ and }\qquad \operatorname{tr}(S_A) = h_A.
\end{equation}
Thanks to \eqref{eq:action curvature Kähler}, \eqref{eq:Laplacian J inv Kähler}, and the Weitzenb\"ock formula \eqref{eq:main-weitzenbock-formula}, if $S_I = \gamma\otimes \omega$ we find:
\begin{equation}\label{eq: weitz J-inv}
\begin{aligned}
     L_f h_I &= \operatorname{tr}\left(\left(\Delta^L_{H, f, 0}+\lambda\right) S_I\right)\\
     &= [\left(\Delta_{H, f, 0} + \lambda\right)\gamma]\circ  \omega.
\end{aligned}
\end{equation}
See Remark \ref{rem:kahler-soliton-conformal} for the derivation of this identity for the conformal part of the deformation $h_I$. 
Similarly, on the $J$-anti-invariant part $h_A$, using \eqref{eq: weitz J-anti}, we find:
\begin{equation}\label{eq: weitz J-anti}
    L_f h_A = \operatorname{tr}\left(\Delta^L_{H, f, 0} S_A\right) + \left(\lambda-\frac{\operatorname{scal}}{2}\right) h_A.
\end{equation}

\subsection{Spectrum of expanding and steady K\"ahler solitons}

Let $(M,g,f,J)$ be a K\"ahler Ricci soliton as above. Consider an arbitrary $2$-tensor $h$. Since $L_f$ preserves $J$-invariant and $J$-anti-invariant $2$-tensors, we find that if $h= h_I+h_A$, then we have
\[
\langle L_f h,h\rangle_{L^2_f} = \langle L_f h_I,h_I\rangle_{L^2_f}+\langle L_f h_A,h_A\rangle_{L^2_f}.
\]
Towards proving that $\langle L_f h,h\rangle_{L^2_f}\leqslant 0$, it is sufficient to show separately that $\langle L_f h_I,h_I\rangle_{L^2_f}\leqslant 0$ and $\langle L_f h_A,h_A\rangle_{L^2_f}\leqslant 0$.

In \eqref{eq: weitz J-inv}, integration by parts gives that $\langle\Delta^L_{H,f,0}S_I,S_I\rangle_{L^2_f}\leqslant 0$. Thus, we find 
\begin{equation}\label{eq: IBP hI}
    \langle L_f h_I,h_I\rangle_{L^2_f} \leqslant \lambda|h_I|^2_{L^2_f}.
\end{equation}
Similarly, thanks to \eqref{eq:action curvature Kähler} and the fact that $S_A \in \Lambda^{1,1}_0 \otimes \mathfrak{R} \Lambda^{2, 0}$, we find that 
\begin{equation}\label{eq: IBP hA}
    \langle L_f h_A,h_A\rangle_{L^2_f} \leqslant \Big(\lambda-\frac{\operatorname{scal}}{2}\Big)|h_A|^2_{L^2_f}.
\end{equation}

\begin{proof}[Proof of Theorem \ref{thm: stab}]
    If $\lambda = 0$, then $\scal\geqslant 0$ (as ancient solutions of Ricci flow, steady solitons always have nonnegative scalar curvature). In this case, \eqref{eq: IBP hI} and \eqref{eq: IBP hA} show that the $L^2_f$-spectrum of $L_f$ is nonpositive. Unless the steady soliton is flat (i.e. $\mathrm{scal} \equiv 0$), elements in the kernel of $L_f$ can only arise via harmonic $(1, 1)$-forms, composed with the K\"ahler form. 
    
    Now for expanding solitons with possibly negative scalar curvature, from \eqref{eq: IBP hI} and \eqref{eq: IBP hA}, it is sufficient to show that $ \lambda-\frac{\operatorname{scal}}{2}<0$. This is the case since as proven in \cite[Theorem 3]{prs11}, $2\lambda\leqslant \operatorname{scal}$ with equality only at Einstein metrics, and of course $\lambda<0$.
\end{proof}

\begin{remark}
    This conclusion is, of course, never possible on shrinking solitons since $\Ric$ is always an eigentensor associated with a positive eigenvalue. 
\end{remark}

\section{Instability of the BCCD soliton}

We briefly review some background. Suppose $(M, g, f)$ is a gradient shrinking Ricci soliton. Recall Perelman's entropy \cite{p02}, 
\[
\nu(g) := \inf \Big\{ \mathcal{W}(g, f, \tau) : \tau > 0, f \in C^{\infty}_0(M)\;\; \text{satisfying}\;\; (4\pi \tau)^{-\frac{n}{2}} \int_M e^{-f} d\mu_g = 1\Big\}.
\]
Metrics of gradient shrinking solitons are critical points of $\nu$, and $t \mapsto \nu(g_t)$ is monotone increasing along a (compact) Ricci flow $g_t$. The second variation of $\nu$ is given by 
\[
\delta^2 \nu_g(h) := \frac{d^2}{ds^2} \nu(g + sh) \big|_{s =0} = (4\pi)^{-\frac{n}{2}} \int_M g(N_f h, h) \, e^{-f} d\mu_g,
\]
where $N_fh := \frac{1}{2} L_fh + \mathrm{div}_f^\ast \mathrm{div}_fh + \frac{1}{2} \nabla^2 v_h - \Xi(h) \Ric$ is the stability operator. The reader may refer to \cite[Section 2.2]{no25} as well as \cite{chi04, cz12, cz24} for further background and notation. In what follows, we only need recall two important points: (1) a gradient Ricci soliton is said to be \textit{linearly unstable} (in the sense of \cite{chi04}) if there exists a deformation $h$ such that $\delta^2 \nu_g(h) > 0$; and (2) if $h \perp_{L^2_f} \Ric$ and $\mathrm{div}_f(h) = 0$, then $\delta^2 \nu_g(h)  = \frac{1}{2} (4\pi)^{-\frac{n}{2}} \langle L_f h, h \rangle_{L^2_f} > 0$.  

Next, we present the core of our argument, which is a noncompact version of the criterion for instability of shrinking Kähler Ricci solitons of \cite[Proposition 4.2]{hm11}.
\begin{proposition}\label{prop: dim H11 instab}
    Let $(M,g,f,J)$ be a shrinking Kähler Ricci soliton, and assume that 
    $$\dim \mathcal{H}^{1,1}_f\geqslant 2.$$
    Then, $(M,g,f,J)$ is linearly unstable in the sense of \cite{chi04}.
\end{proposition}
\begin{remark}
    On the blowdown shrinking soliton of \cite{fik}, one has $\dim \mathcal{H}^{1,1}_f = 1$, and $\mathcal{H}^{1,1}_f$ spanned by the Ricci curvature, so the criterion does not apply. 
\end{remark}
\begin{proof}
    Let $\gamma\in \Lambda^{1,1}$ be a solution of $d\gamma = 0$ and $\delta_{f,0}\gamma=0$. Assume that $\gamma\perp_{L^2_f}\rho$ where $\rho$ is the Ricci form so that the symmetric $2$-tensor $\operatorname{tr}(\gamma\otimes\omega)$ is $L^2_f$-orthogonal to $\Ric(g) = \operatorname{tr}(\rho\otimes\omega)$.

    Then, if $\gamma\in H^1_f$, then, from \eqref{eq: weitz J-inv} which follows closely a computation of \cite{hm11} the $2$-tensor $h=h_I:=\gamma\circ\omega$ satisfies $h\perp_{L^2_f} \Ric(g)$, $\operatorname{div}_fh=0$ and
    \[
    \langle L_fh,h\rangle = \big\langle \left(\Delta_{H,f,0}+\lambda\right)\gamma\,,\,\gamma\,\big\rangle |\omega|^2  = \lambda \,|h|^2 >0,
    \]
    which concludes the proof. 
    The above deformation $h$ is also an eigentensor of $L_f$ with eigenvalue $\lambda=1$, so $t\mapsto e^{t}h$ solves the heat equation $(\partial_t-L_f)(e^{t}h) = 0$ and is uniformly bounded on $(-\infty,0]$.
\end{proof}

By Proposition \ref{prop: dim H11 instab}, in order to prove the instability of the soliton of \cite{bccd}, it is sufficient to construct a $2$-dimensional set of weighted harmonic $(1,1)$-forms. Towards this, the first step consists in fixing two independent cohomology classes. 

To that end, let $(M^4, J) := \big(\operatorname{Bl}_1(\mathbb{C}\times \mathbb{CP}^1), J \big)$ denote our complex surface with its standard complex structure. Suppose the blow-up is done at the point $p_0:= (0, p)$ and let $q \in \mathbb{CP}^1 \setminus \{p\}$ be a distinct point. Let $C_z := \{z\} \times \mathbb{CP}^1 \subset\mathbb{C}\times \mathbb{CP}^1$, and let $\pi : M^4 \to \mathbb{C} \times \mathbb{CP}^1$ denote the projection. Define 
\[
\Sigma_1 := \pi^{-1}(p_0), \qquad  \Sigma_2 := \pi^{-1}(C_1), \qquad \tilde{\Sigma} := \overline{\pi^{-1}(C_0 \setminus \{p_0\})}, \qquad \mathbb{C}_q := \pi^{-1} (\mathbb{C} \times \{q\})
\]
Then $\Sigma_1, \Sigma_2, \tilde{\Sigma}$ are each a holomorphic $\mathbb{CP}^1$ with self-intersection numbers $-1, 0, -1$, respectively. $\Sigma_1$ is the exceptional divisor and $\tilde{\Sigma}$ is the proper transform of $C_0$. On the other hand $\mathbb{C}_q$ is a complex line which does not intersect $\tilde{\Sigma}$.  

Let $H^{1,1}_0(M)$ denote the space of $(1, 1)$-forms on $M$ with compact support. The following lemma is a consequence of the far-reaching Lefschetz theorem. 

\begin{lemma}\label{lem: cohomology bccd}
    There exist two compactly-supported, closed, $(1,1)$-forms, $\tilde{\gamma}_1$ and $\tilde{\gamma}_2$, with disjoint support such that $\int_{\tilde{\Sigma}} \tilde\gamma_1= 1$ and $\int_{\tilde{\Sigma}} \tilde{\gamma}_2 = 0$, while $\int_{\mathbb{C}_q} \tilde{\gamma}_1 = 0$ and $\int_{\mathbb{C}_q} \tilde{\gamma}_2 = 1$. In particular, $\tilde{\gamma}_1, \tilde{\gamma}_2$ are independent, nontrivial elements $H^{1,1}_0(M)$.  
\end{lemma}

The basic idea is to consider the Thom classes associated to the normal bundles of the surfaces $\Sigma_1$ and $\Sigma_2$, respectively. In particular, tubular neighborhoods of $\Sigma_1$ and $\Sigma_2$ in $M$ are biholomorphic to tubular neighborhoods of complex line bundles $\mathcal{O}(-1) \to \mathbb{CP}^1$ and $\mathcal{O}(0) \to \mathbb{CP}^1$. The (compactly-supported) Thom classes associated to these bundles are transferred to $M$ via the biholomorphisms. The nontriviality and independence of the forms obtained follows from the evident intersection properties of $\Sigma_1, \Sigma_2, \tilde{\Sigma}$ and $\mathbb{C}_q$. 

An explicit description of the Thom class for an oriented real plane bundle $E \to \Sigma$ may be found in a well written note of Nicolaescu \cite[Section 5]{nic11}, based on the excellent text of Bott and Tu \cite[Chapter 6]{bt82}. (In particular, see Proposition 6.24 and Proposition 6.25 in \cite{bt82}.) In our complex setting, as $\Sigma_1, \Sigma_2$ are complex surfaces in $(M, J)$, the forms representing the Thom classes obtained are necessarily $(1, 1)$.

We finally prove our main theorem.

\begin{proof}[Proof of Theorem \ref{thm: instab BCCD}]
    Let $(M,g,f,J)$ be the shrinking Kähler Ricci soliton of \cite{bccd}. We will now show that the space of weighted-harmonic $(1,1)$-forms $$\mathcal{H}_f^{1,1}:=\big\{\gamma\in \Lambda^{1,1}\cap H^1_f,\,\, d\gamma = 0, \text{ and } \delta_{f,0} \gamma=0\big\},$$
    which contains the Ricci form $\rho$ is at least $2$-dimensional. Together with Proposition \ref{prop: dim H11 instab}, this will show Theorem \ref{thm: instab BCCD}. 
    
    Consider, for $i\in\{1,2\}$, the closed compactly-supported $(1,1)$-forms $\tilde{\gamma}_i$ obtained from Lemma \ref{lem: cohomology bccd}. We now search for $\gamma_i$ in the cohomology class $ [\tilde{\gamma}_i] := \{\tilde{\gamma}_i + d \eta, \eta\in \Omega^1(M)\} $. The equation $d\gamma_i = 0$ is automatically satisfied, so we will focus on the other, and search for $\eta\in H^3_f$ so that  
    $$ \delta_{f,0}\gamma_i = \delta_{f,0}\tilde{\gamma}_i + \delta_{f,0}d\eta = 0.$$
    Now, as noted in \cite{hm11}, $\Delta_{H, f,0}$ preserves the decomposition in $(p,q)$-forms, so $\Delta_{H, f,0} \tilde{\gamma}_i$ is a $(1,1)$-form and $\Delta_{H, f,0}d\eta = -\Delta_{H, f,0} \tilde{\gamma}_i$ is automatically $(1,1)$. It is not directly clear that $\gamma_i$ is automatically $(1,1)$ by this construction, but its projection $\pi_{(1,1)}\gamma_i$ will solve our equations. 
    
    Since $\gamma_i = \tilde{\gamma}_i + d \eta$, we may additionally assume that $\delta_{f,0}\eta = 0$ without changing the value of $\gamma_i$ given that it amounts to adding $d\phi$ to $\eta$ for some function $\phi\in H^4_f$. Indeed, this is equivalent to solving $$-\Delta_f\phi = \delta_{f,0} d \phi = - \delta_{f,0}\eta.$$ This can always be solved since $\Delta_f$ has discrete spectrum and $\delta_{f,0}\eta$ is $L^2_f$-orthogonal to the cokernel, $\mathbb{R} \cdot 1$, because $\int_M\delta_{f,0}\eta \,\,e^{-f}dv = 0$ by integration by parts, since we assumed $\eta\in H^3_f$.

    Now, assuming $\delta_{f,0}\eta = 0$, our equation rewrites: 
    \begin{equation}\label{eq: 1 form BCCD}
        (\delta_{f,0}d+d\delta_{f,0})\eta = - \delta_{f,0}\tilde{\gamma}_i\in H^1_f.
    \end{equation}
    Let us show that this equation can indeed be solved. We will make use of the identity: on $1$-forms
    \[
    (\delta_{f,0}d+d\delta_{f,0}) = -\Delta_f + \Ric_f = -\Delta_f + \frac{1}{2}. 
    \]
    Now, working with $-\Delta_f$ on $1$-forms exactly as \cite{cm21} do on vector fields in Lemma 4.17, we see that $-\Delta_f$ has discrete \textit{nonnegative} spectrum on $1$-forms.
    
    Consequently, the operator $\delta_{f,0}d+d\delta_{f,0}$ also has discrete spectrum, and its eigentensors are the same as those of the operator $-\Delta_f$, with associated eigenvalues only shifted by $\frac{1}{2}$. In particular, $\delta_{f,0}d+d\delta_{f,0}$ is invertible. More concretely, let us denote $v_i\in H^1_f$ an $L^2_f$-orthonormal basis of eigentensors with eigenvalues $\mu_i>\frac{1}{2}$. One can decompose $- \delta_{f,0}\tilde{\gamma}_i = \sum_{i} b_i v_i\in H^1_f$ and solve \eqref{eq: 1 form BCCD} with $\eta = \sum_i \frac{b_i}{\mu_i}v_i\in H^1_f$. 
    \\

    A posteriori, applying $\delta_{f,0}$ to both sides of \eqref{eq: 1 form BCCD}, we see that $\delta_{f,0}\eta = 0$ is satisfied by a solution, and elliptic theory implies that $\eta\in H^3_f$. Noting that the $(1,1)$-forms $\tilde{\gamma}_1$ and $\tilde{\gamma}_2$ were non-homologous to begin with, we then obtain two independent elements $\pi_{(1,1)}\gamma_1$ and $\pi_{(1,1)}\gamma_2$ of $\mathcal{H}^{1,1}_f$, that is $\dim\mathcal{H}^{1,1}_f\geqslant2$ as desired.
\end{proof}

\section{Orbifold stability of Kähler solitons}

As suspected in \cite{bam1,bam2,bam3}, limits of specific blowups of singularities should be \textit{orbifold} shrinking Ricci solitons. The stability of an orbifold singularity at a point $p_0$ is determined by the eigenvalues of weighted selfdual curvature, $\overline{\mathbf{R}}{}^+_{p_0}$ at the point $p_0$ as explained in \cite{doorb}. If $\overline{\mathbf{R}}{}^+_{p_0}$ has one positive eigenvalue, then the singularity is \textit{unstable} while if all of the eigenvalues of $\overline{\mathbf{R}}{}^+_{p_0}$ are negative, the singularity is \textit{stable}. There are many examples of Einstein metrics with orbifold singularities, but relatively few examples are known in the shrinking case. Some examples include specific quotients of cylinders or the compact orbifold shrinking soliton of \cite{fik}.

We now consider an orbifold Kähler soliton and discuss whether they are stable or unstable in the sense of \cite{doorb}. This is luckily a very simple stability criterion which only requires knowing the curvature at the singular point of the manifold. This is particularly simple on Kähler manifolds.

With the conventions of the present article, the notion of \textit{orbifold point stability} of \cite{doorb} is determined by the sign of the eigenvalues of the \textit{weighted selfdual curvature} 
$$ \overline{\mathbf{R}}{}^+ := \frac{\operatorname{scal}}{6}\operatorname{Id}_{\Lambda^+_g} + W^+ +\frac{\Delta f}{2}\operatorname{Id}_{\Lambda^+_g} = \left(\lambda - \frac{\operatorname{scal}}{3}\right)\operatorname{Id}_{\Lambda^+_g}+ W^+, $$
of which a generalized form appears in our Weitzenböck formulae in Appendix \ref{app:weitzenbock}, see Remark \ref{rem: overline R} with $a=0$.

\begin{proposition}\label{prop: instab non kahler}
     All orbifold shrinking solitons have unstable singularities.
\end{proposition}
\begin{proof}
    The trace of $\overline{\mathbf{R}}{}^+$ is $\frac{\operatorname{scal}}{2} + \frac{3}{2}\Delta f = 3\lambda-\operatorname{scal}$. By \cite[Theorem 3]{prs11}, in dimension $4$, $3\lambda-\operatorname{scal}>0$, so at least one of the eigenvalues of $\overline{\mathbf{R}}{}^+$ is positive. In other words, the orbifold singularities are unstable. 
\end{proof}

\begin{remark}
    In the Kähler case, it turns out that an unstable direction is given by the \textit{Kähler form} itself. It motivates Conjecture \ref{conj: ancient KRF orb}.
\end{remark}

\begin{proof}[Proof of Theorem \ref{thm: stab orb}]
    The result follows from combining \eqref{eq:action curvature Kähler} and \cite[Theorem 3]{prs11}.
\end{proof}

\appendix

\section{Table of Solitons}\label{app:solitons}

A table of \textit{known, smooth, simply-connected, gradient} Ricci shrinking solitons in dimension four:
\begin{table}[h]
    \centering
    \begin{tabular}{|c|c|c|c|c|c|c|c|}
        \hline
        \textbf{Name (\& Symmetry)} & \textbf{Topology} & \textbf{K} & \textbf{E} & \textbf{P} & \textbf{$\Theta$} & \textbf{Stab.}   &\textbf{Stab. Ref.} \\ 
        \hline
        Gaussian ($E_4$) &  $\mathbb{C}^2$ & K & & P & 1 &   {\color{teal}S} & - \\ 
        \hline
        Sphere ($O_5$) &  $\mathbb{S}^4$ & & E & &0.8120  & {\color{teal}S} & \cite{ham86} \\ 
        \hline
        Cylinder ($O_4\times E_1$) & $\mathbb{S}^3\times \mathbb{R}$ & & & P &  0.7910 &  {\color{teal}S} & -   \\ 
        \hline
        Cylinder ($O_3 \times E_2$) & $\mathbb{CP}^1 \times \mathbb{C}$ & K & & P &  0.7358& {\color{teal}S}& - \\ 
        \hline
        FIK Blowdown ($U_2$) & $\mathrm{Bl}_1\mathbb{C}^2$ & K &  &  & 0.6720 &{\color{teal}S}& \cite{no25}\\
        \hline
        Fubini-Study ($U_2$) &  $\mathbb{CP}^2$ & K& E & & 0.6090 &   {\color{purple}U}&  \cite{kro15} \\ 
        \hline
        BCCD ($O_2 \times O_2$)& $\mathrm{Bl}_1(\mathbb{CP}^1 \times \mathbb{C})$& K & & & 0.5617  &  {\color{purple}U} & this paper\\ 
        \hline
        -  ($O_3 \times O_3$)  &  $\mathbb{CP}^1 \times \mathbb{CP}^1$ & K& E& P & 0.5413 & {\color{purple}U}& \cite{chi04} \\ 
        \hline
        Koiso-Cao ($U_2$) & $\mathrm{Bl}_1\mathbb{CP}^2 $ &  K  & & &0.5179 & {\color{purple}U}& \cite{hm11} \\ 
        \hline 
        Page ($U_2$) & $\mathrm{Bl}_1\mathbb{CP}^2 $  & &  E &  & 0.5172 & {\color{purple}U}& \cite{hhs14} \\ 
        \hline 
        Chen-Lebrun-Weber  ($O_2 \times O_2$) & $\mathrm{Bl}_2\mathbb{CP}^2 $ & &  E &  & 0.4552 & {\color{purple}U} & \cite{bo23}\\
        \hline 
        Wang-Zhu ($O_2 \times O_2$) &  $\mathrm{Bl}_2\mathbb{CP}^2 $ & K & & & 0.4549 & {\color{purple}U} & \cite{hm11}\\ 
        \hline 
        Blowups of $\mathbb{CP}^2$, $3 \leq k \leq 8$ (-)  &$\mathrm{Bl}_k\mathbb{CP}^2 $  & K & E &  & $<$ 0.406 & {\color{purple}U} & \cite{chi04} \\ 
        \hline
    \end{tabular}
    \caption{K $=$ K\"ahler, E $=$ Einstein, P $=$ Product, $\Theta = $  Approx. Central Density, {\color{teal} S $=$ Stable}, {\color{purple} U $=$ Unstable}.}
    \label{tab:page13}
\end{table}

We offer several remarks to go along with the table:
\begin{enumerate}[(i)]
    \item A smooth K\"ahler shrinker must be simply-connected \cite{esp25} and must have uniformly bounded curvature \cite{lw23}. All K\"ahler shrinking solitons have been classified \cite{bccd,lw23,cds24}.
    \item There are stable quotient solitons (such as $\mathbb{C}^2/\mathbb{Z}_2, \mathbb{RP}^4, \mathbb{RP}^3 \times \mathbb{R}$), which have smaller central density than some simply-connected examples (half those of $\mathbb{C}^2, \mathbb{S}^4$ and $\mathbb{S}^3 \times \mathbb{R}$, respectively). In general, if $(M, g, f)$ is linearly stable, then so is its universal cover $(\tilde{M}, \tilde{g}, \tilde{f})$ with the pullback metric $\tilde{g}$ and potential $\tilde{f}$.  
    \item Tangent flows (suitably-defined) of compact 4-dimensional Ricci flows yield either a smooth gradient shrinking soliton, or else a gradient shrinking soliton with finitely many orbifold singularities, \cite{bam1, bam2, bam3}. The latter case includes, for example, the flat $\mathbb{C}^2/\mathbb{Z}_2$. Besides flat examples, all currently known linearly stable (with respect to orbifold deformations) examples are smooth. All singular examples are unstable with respect to desingularizations by known Ricci-flat ALE spaces by Proposition \ref{prop: instab non kahler}.  
    \item The Fubini-Study metric is \textit{linearly} (semi)stable (also called weakly stable), but \textit{dynamically} unstable \cite{kro15}. Indeed, a Taylor expansion of Perelman's entropy at $\mathbb{CP}^2$ has nonvanishing third-order term. The instability was independently obtained by Knopf-\v Se\v sum \cite{ks19}. See also the Hamilton's analysis in Section 10 of \cite{ham95}.
    \item Values and approximate values for the central densities can be found in \cite{chi04} and \cite{hal11}. Using the approach taken by Hall in \cite[Section 3]{hal11} and the description of the BCCD soliton as a toric Kähler Ricci soliton in \cite[Example 2.33]{ccd24}, one may show that the central density of the BCCD soliton is approximately 0.5617, as indicated in Table \ref{tab:page13}. Specifically, the density may be computed as the minimum a weighted volume functional $F$ over the moment polytope $P$ given here by
    \begin{align*}
   & \Theta(g_{\mathrm{BCCD}}) = e^{-2} \min_{c \in \mathbb{R}} F(c), \qquad F(c) := \int_P e^{-c (2x_1 + x_2)}, \\
    P&:=\{ (x_1, x_2) \in \mathbb{R}^2 : x_1\geq -1\,,\, 1 \geq x_2 \geq -1\,,\, x_1 + x_2 \geq -1\}.
    \end{align*}
    The minimizing $c$ occurs approximately at $c \approx 0.6438$.
    \item Topologically, one has $\mathrm{Bl}_k X = X \# k\,\overline{\mathbb{CP}^2}$ as differentiable manifolds.
    \item For $k \in \{5, 6, 7, 8\}$, the shrinking solitons on $\mathrm{Bl}_k \mathbb{CP}^2$ come in infinitely families (of dimensions $2,4,6,8$, respectively) which are not bi-holomorphic, but are isometric (see the introduction of \cite{lw23}). 
    \item The Page and Chen-Lebrun-Weber metrics are conformally K\"ahler. 
    \item In the (non-K\"ahler) Riemannian setting, some of the solitons above are usually written: 
    \[\mathbb{R}^4 = \mathbb{C}^2, \qquad \mathbb{S}^2\times \mathbb{R}^2 = \mathbb{CP}^1 \times \mathbb{C}, \qquad \mathbb{S}^2\times \mathbb{S}^2 = \mathbb{CP}^1 \times \mathbb{CP}^1.
    \]
    \item The named solitons in the list were discovered in \cite{p79}, \cite{koi90}, \cite{cao96}, \cite{fik} \cite{wz04},  \cite{clw08}, and \cite{bccd}. 
\end{enumerate}

\section{Weitzenb\"ock Formulae}\label{app:weitzenbock}

In this section, we collect conventions and give the derivation that leads to Proposition \ref{prop:main-weitzenbock}. Along the way, we derive some classical Weitzenb\"ock formulae and a family of lesser-known weighted formulae. We focus on dimension four with the intention of using these formulae on gradient Ricci solitons. 
\subsection{Conventions and classical formulae}\label{sec:classical-weitzenbock} 
Throughout this section, we suppose $(M^4, g)$ is an oriented $4$-dimensional Riemannian manifold and $f \in C^{\infty}(M)$ is a smooth function. Let $\{e_0, e_1, e_2, e_3\}$ be a local orthonormal frame and $\{e^0, e^1, e^2, e^3\}$ its dual on $M$. We use index notation with respect to this frame and adopt the Einstein summation convention for repeated indices. 

\subsubsection{On forms}
In what follows, $d$ and $\delta$ denote exterior derivative and codifferential, and $\Delta_H$ the Hodge Laplacian. These act on a form $\alpha$ by 
\[
d\alpha = e^i \wedge (\nabla_{e_i} \alpha), \qquad \delta \alpha =  -\,\iota_{e_i} (\nabla_{e_i} \alpha), \qquad \Delta_H \alpha  = -(d \delta + \delta d)\alpha.
\]
We adopt the normalized wedge convention,  $\alpha \wedge \beta = \frac{k!\,l!}{(k+1)!}\operatorname{Alt}(\alpha \otimes \beta)$ for $\alpha\in \Lambda^k, \beta \in \Lambda^l$, so that $\iota_X \alpha(\cdots) = k\, \alpha(X,\cdots)$. This convention implies $|e^{i_1} \wedge \cdots \wedge e^{i_k}|^2 = \frac{1}{k!}$; for instance, $e^i \wedge e^j = \frac{1}{2}(e^i \otimes e^j - e^j \otimes e^i)$. To ensure $d$ and $\delta$ are $L^2$-adjoint, as well as $e^i \wedge$ and $\iota_{e_i}$, we use the scaled inner product $\hat{g} = k! g$ on $\Lambda^k$, so that $e^{i_1} \wedge \cdots \wedge e^{i_k}$ are orthonormal. With $\alpha \in \Lambda^k$, and $\beta \in \Lambda^{k+1}$, we then have
\[
\hat g(d\alpha, \beta) -\hat g(\alpha, \delta \beta) = (k+1) \nabla_{e_i}\,( \hat{g}(\alpha, \iota_{e_i}\beta)) = (k+1)\mathrm{div}(\hat{g}(\alpha, \iota_{(\cdot)}\beta)).
\]
In particular, formally,
\[
\int_M \hat{g}(\Delta_H \alpha, \alpha ) d\mu_g = -\int_M |d\alpha|^2_{\hat g} + |\delta \alpha|_{\hat{g}}^2 \;d\mu_g. 
\]

From these definitions, one obtains the classical Weitzenb\"ock formula $\Delta_H \alpha = \Delta \alpha + e^i \wedge \iota_{e_j} [R(e_i, e_j) \alpha]$, where $R(e_i,e_j) \alpha = (\nabla_{e_j} \nabla_{e_i} -\nabla_{e_i} \nabla_{e_j} - \nabla_{[e_j, e_i]} )\alpha$. We will only use this formula on 2-forms.  To state it more concretely, recall the Kulkarni-Nomizu product acting on symmetric 2-tensors $A, B$ is given by 
\begin{equation}\label{eq:kn-formula}
(A \owedge B)_{ijkl} = A_{ik} B_{jl}  - A_{il} B_{jk} - A_{jk} B_{il} + A_{jl} B_{ik}. 
\end{equation}
This convention implies that $\mathrm{Id}_{\Lambda^2} = \frac{1}{4} g \owedge g$. 
The Ricci decomposition formula in dimension four is then given by 
\begin{equation}\label{eq:ricci-decomp}
R  = \frac{1}{2} \Ric \owedge g +  W - \frac{\scal}{3} \mathrm{Id}_{\Lambda^2} =\frac{1}{2} \mathring{\mathrm{Ric}} \owedge g +  W + \frac{\scal}{6} \mathrm{Id}_{\Lambda^2}.
\end{equation}
Finally, a section $T$ of $\Lambda^2 \otimes \Lambda^2 $ acts on a 2-form by $T(\phi)_{ij} = T_{ijkl} \phi_{kl}$. 
From these, a computation readily implies the following classical formula: 
\begin{proposition}\label{prop:weitzenbock-classic}
Suppose $\phi$ is a $2$-form on $(M^4, g)$. Then 
\[
\Delta \phi =  \Delta_H \phi -  \big(W - \frac{\scal}{3} \mathrm{Id}_{\Lambda^2}\big)(\phi).
\]
\end{proposition}

\subsubsection{On bi-forms}\label{sec:bi-forms}
Given any vector bundle $E$ over $M$ equipped with a metric and connection, there are natural extensions $d^E$ and $\delta^E$ of the exterior derivative and codifferential to $E$-valued $k$-forms, i.e. sections of $\Lambda^k \otimes E$. If one takes $E = \Lambda^p$ with the natural extensions of the metric $g$ and Levi-Civita connection $\nabla$, one obtains left and right acting differentials and co-differentials. Given $\alpha \in \Lambda^k$ and $\beta \in \Lambda^l$, these operators act according to the formulas 
\begin{align*}
d^L (\alpha \otimes \beta) &= (d\alpha) \otimes \beta + (-1)^k (\alpha \wedge \nabla \beta), 
& d^R(\alpha \otimes \beta) &= \alpha \otimes (d \beta) +(-1)^l(\nabla\alpha \wedge \beta) \\
\delta^L (\alpha \otimes \beta) &= (\delta \alpha) \otimes \beta - \mathrm{tr}_g(\iota_{(\cdot)} \alpha \otimes \nabla_{(\cdot)}\beta) , 
& \delta^R(\alpha \otimes \beta) &= \alpha \otimes (\delta \beta) - \mathrm{tr}_g(\nabla_{(\cdot)}\alpha \otimes \iota_{(\cdot)}  \beta),
\end{align*}
where
\[
\alpha \wedge \nabla \beta :=\sum_{i =1}^n (\alpha \wedge e^i)\otimes \nabla_{e_i} \beta, \qquad \nabla \alpha \wedge \beta := \sum_{i =1}^n \nabla_{e_i} \alpha \otimes (\beta \wedge e^i),
\]
and 
\[
\mathrm{tr}_g\big(\iota_{(\cdot)} \alpha \otimes \nabla_{(\cdot)}\beta\big):= \sum_{j =1}^n ((\iota_{e_j}\alpha) \otimes \nabla_{e_j}\beta), \qquad \mathrm{tr}_g\big(\nabla_{(\cdot)} \alpha \otimes \iota_{(\cdot)}\beta\big):= \sum_{j =1}^n ((\nabla_{e_j}\alpha) \otimes \iota_{e_j}\beta). 
\]
To give a concrete example, suppose $S = \{S_{ijpq}\}$ is a section of $\Lambda^2 \otimes \Lambda^2$. Then $d^L S$ is a section of $\Lambda^3 \otimes \Lambda^2$, $\delta^L S$ a section of $\Lambda^1 \otimes \Lambda^2$, $d^R S$ a section of $\Lambda^2 \otimes \Lambda^3$, and $\delta^R S$ a section of $\Lambda^2 \otimes \Lambda^1$ each given by 
\begin{align*}
(d^L S)_{ijkpq} &= \frac{1}{3} \big(\nabla_i S_{jkpq} + \nabla_j S_{kipq} + \nabla_k S_{ijpq}\big),& (d^R S)_{ijpqr} &= \frac{1}{3} \big(\nabla_r S_{ijpq} + \nabla_p S_{ijqr} + \nabla_q S_{ijrp}\big)\\
(\delta^L S)_{ipq} &= -2 \nabla_k S_{kipq}, & (\delta^R S)_{ijp} &= -2 \nabla_q S_{ijqp}.
\end{align*}
With the differentials and co-differentials above, we define the left and right Hodge Laplacians
\[
\Delta^L_H = - (d^L \delta^L + \delta^L d^L), \qquad \Delta^R_H = - (d^R \delta^R + \delta^R d^R).
\]
As above, use the scaled inner product 
\[
\hat{g} = k!\, l! \,g  \quad \text{on} \quad \Lambda^k \otimes \Lambda^l,
\]
with respect to which these Laplacians are self adjoint. 
The space $\Lambda^2 \otimes \Lambda^2$ comes with two natural operations, 
\begin{align*}
(S \circ T)_{ijkl} &= S_{ijpq} T_{pqkl}, \\
(S\# T)_{ijkl} &= S_{ipkq} T_{jplq} - S_{iplq} T_{jpkq} - S_{jpkq} T_{iplq} + S_{jplq} T_{ipkq}.
\end{align*}
for any sections $S$ and $T$. Note that $S \# T = T\#S$. We define $S^2 := S \circ S$ and $S^\# := S \# S$. The operation $\#$ has an interpretation in terms of the Lie bracket on $\mathrm{so}(n)$. Indeed, on 2-forms $\phi, \psi$, the Lie bracket is a 2-form given by $[\phi, \psi]_{ij} = \phi_{ik} \psi_{jk} - \psi_{ik} \phi_{jk}$. Then
\begin{align}\label{eq:sharp-lie}
(\phi \otimes \tilde{\phi}) \# (\psi \otimes \tilde{\psi}) & = [\phi, \psi] \otimes [\tilde{\phi}, \tilde{\psi}]. 
\end{align}

With these conventions and notation established, we may now state our second Weitzenb\"ock formula for sections of $\Lambda^2 \otimes \Lambda^2$. 

\begin{proposition}\label{prop:weitzenbock-bi-forms}
Suppose $S$ is a section of $\Lambda^2 \otimes \Lambda^2$ on $(M^4, g)$. Then 
\begin{align*}
\Delta S = \Delta^L_H S -  \big( W-\frac{\scal}{3} \mathrm{Id}_{\Lambda^2}\big) \circ S - R \# S, \\
\Delta S  =  \Delta^R_H S - S \circ \big( W-\frac{\scal}{3} \mathrm{Id}_{\Lambda^2}\big) -  S \# R.
\end{align*}
\end{proposition}

\subsection{Weighted formulae}\label{sec:weighted-operators}

In the context of Ricci solitons, weighted Weitzenb\'ock formula, i.e. those involving the weighed Laplacian $\Delta_f = \Delta - \nabla_{\nabla f}$ are more natural. To describe them, we introduce a family of weighted differentials for $a \in [0, 1]$, among which the operators for $a = 0$ (\cite{hm11}) and $a = \frac{1}{2}$ (the balanced one) are most important. 

Let $f \in C^{\infty}(M)$. We define
\[
d_{f, a}(\cdot) := e^{af} d(e^{-af}\, \cdot \,  ), \qquad \delta_{f, a}(\cdot) := e^{(1-a)f} \delta(e^{(a-1)} \, \cdot \, ), \qquad \Delta_{H, f, a}  = - (d_{f, a} \delta_{f, a} + \delta_{f, a} d_{f, a}). 
\]
These formulas imply that $d_{f, a} = d - a df \wedge$ and $\delta_{f, a} =  \delta + (1- a) \iota_{\nabla f}$. We similarly define $d^L_{f, a} ,\delta^L_{f, a}, \Delta^L_{H, f, a}$, as well as $d^R_{f, a}, \delta^R_{f, a}, \Delta^R_{H, f, a}$. 

A straightforward computation yields the following lemma, which relates the weighted formulas to the classical ones.  

\begin{lemma}\label{lem:weighted-with-a}
On a 2-form $\phi$,
\begin{align*}
\Delta_H \phi -\nabla_{\nabla f} \phi    & =  \Delta_{H, f, a}\phi -  (a-\frac{1}{2}) \big(\mathrm{Hess} f \owedge g\big)(\phi) + a \Big(\Delta f -(1-a) |\nabla f|^2\Big) \phi.
\end{align*}
On a section of $S$ of $\Lambda^2 \otimes \Lambda^2$,
\begin{align*}
\Delta^L_H S  -(\nabla_{\nabla f} S) &=  \Delta^L_{H, f, a} S   - (a-\frac{1}{2}) \big(\mathrm{Hess} f \owedge g\big) \circ S + a \Big(\Delta f -(1-a) |\nabla f|^2\Big) S, \\
 \Delta^R_H S  -(\nabla_{\nabla f} S)  &=\Delta^R_{H, f, a} S  - (a-\frac{1}{2})S \circ \big(\mathrm{Hess} f \owedge g\big)  + a \Big(\Delta f -(1-a) |\nabla f|^2\Big) S.
\end{align*}
\end{lemma}
For every $a$, one has $d_{f, a}$ and $\delta_{f, a}$ are formally adjoint in the weighted $L^2$-space, $L^2_f = L^2(e^{-f} d\mu_g)$. In particular, for $\alpha \in \Lambda^k$ and $\beta \in \Lambda^{k+1}$, one has 
\[
\hat{g}(d_{f, a} \alpha, \beta) e^{-f} - \hat{g}(\alpha, \delta_{f, a} \beta)e^{-f}  =(k+1) \mathrm{div}_f\big(\hat{g}(\alpha, \iota_{(\cdot)}\beta)\big) e^{-f},
\]
where $\mathrm{div}_f(\cdot) = e^f \mathrm{div}(e^{-f}\cdot)$ Similar formulas yield the adjoint properties of the left and right operators in bi-forms. For example, formally, 
\[
\int_M \hat{g}(\Delta^L_{H, f, a} S, S ) \, e^{-f} d\mu_g = - \int_M \big(|d^L_{f,a} S|^2_{\hat{g}} + |\delta^L_{f,a} S|^2_{\hat{g}}\big) e^{-f} d\mu_g. 
\]

\subsection{The trace map and $2$-tensors}\label{sec:trace}
The Hodge star is defined by the identity 
\[
\alpha \wedge \ast \beta = \hat{g}(\alpha, \beta) \mu_g,
\]
where $\mu_g = e^0 \wedge e^1 \wedge e^2 \wedge e^3$ in a local, oriented, orthonormal frame. Using the scaled metric $\hat{g}$ ensures, for instance, that $\ast e^0 =  e^1 \wedge e^2 \wedge e^3$ and $\ast (e^0 \wedge e^1) = e^2 \wedge e^3$. In dimension four, the Hodge star yields a decomposition of $2$-forms into self-dual ($\ast \alpha = \alpha$) and anti-self-dual 2-forms ($\ast \alpha = - \alpha$), giving $\Lambda^2 = \Lambda_g^+ \oplus \Lambda_g^-$. The oriented local frames gives natural bases
\begin{align*}
    \omega_1^{\pm} = e^0 \wedge e^1 \pm e^2 \wedge e^3, \qquad \omega_2^{\pm} = e^0 \wedge e^2 \pm e^3 \wedge e^1 , \qquad \omega_3^{\pm} = e^0 \wedge e^3 \pm e^1 \wedge e^2.  
\end{align*}
Note our conventions imply $|\omega_a^{\pm}| = 1$, or equivalently, $|\omega_a^\pm|_{\hat{g}} = \sqrt{2}$. 

On $\Lambda^2 \otimes \Lambda^2$, the one naturally has a left-acting Hodge star $\ast^L$ and a right acting Hodge star $\ast^R$, and from this one obtains a well-known decomposition 
\[
\Lambda^2 \otimes \Lambda^2 = (\Lambda^+_g \otimes \Lambda^+_g )\oplus (\Lambda^+_g \otimes \Lambda^-_g )\oplus (\Lambda^-_g \otimes \Lambda^+_g )\oplus (\Lambda^-_g \otimes \Lambda^-_g ). 
\]
It is useful to consider the curvature tensor in this decomposition. In particular, noting \eqref{eq:ricci-decomp}, we recall that (because $\ast W = W \ast$) the Weyl component is the sum of self-dual and anti-selfdual parts $W = W^+ + W^-$, and that traceless Ricci component, $\mathring{\Ric} \owedge g$, lies in $(\Lambda^-_g \oplus \Lambda^+_g)\oplus (\Lambda^+_g \otimes \Lambda^-_g)$. 

A key observation in dimension four (e.g. in the study of deformations of metrics), going back to Besse \cite{bes87}, is that there is a bijection, 
\begin{equation}
    \Lambda_g^- \otimes \Lambda_g^+ \cong \mathrm{Sym}^2_0\; T^\ast M,
\end{equation} 
between pairs of self-dual and anti-self-dual 2-forms and traceless symmetric 2-tensors.   In general, given a section $S_{ijkl}$ of $\Lambda^2 \otimes \Lambda^2$, one defines a symmetric $2$-tensor by 
\begin{equation}
\mathrm{tr}_g(S)_{ik} := \frac{1}{2}(S_{ipkp} + S_{kpip}).
\end{equation}
We will sometimes write $\mathrm{tr}$ instead of $\mathrm{tr}_g$ when the metric is understood. When $S = \alpha \otimes \beta$, we sometimes denote this map (abusing notation) by 
\begin{equation}
\alpha \circ \beta := \mathrm{tr}_g(\alpha \otimes \beta). 
\end{equation}
By orthogonality of self-dual and anti-self-dual forms, the trace maps a section of $\Lambda_g^- \otimes \Lambda_g^+$ to a traceless symmetric $2$-tensor. This is a bijection, with inverse given by 
\begin{equation}
\mathring{h} \mapsto \frac{1}{2} \pi_{-, +} (\mathring{h} \owedge g),
\end{equation}
where $\pi_{-,+} = \frac{1}{4}(1-\ast^L)(1+\ast^R)$ is the projection from $\Lambda^2 \otimes \Lambda^2$ to $\Lambda_g^-\otimes \Lambda_g^+$. Formulas for $\omega_a^- \circ \omega_b^+$ can be found in \cite[Section 2.3]{no25}. For $S, T \in \Lambda^-_g \otimes \Lambda^+_g$, we have
\[
g\big(\mathrm{tr}(S), \mathrm{tr}(T)\big) = \frac{1}{4} g(S, T) = \frac{1}{16}\hat{g}(S, T).
\]
This can be checked pointwise in the duality basis $\omega_a^- \circ \omega_b^+$. 

For application to the study of stability of Einstein manifolds, one seeks to relate the stability operator $L = \Delta + 2R$ on traceless 2-tensors, to the Hodge Laplacian on sections of $\Lambda^-_g \otimes \Lambda^+_g$. In the following proposition, we derive this formula. It is usually applied assuming that $\mathring{\Ric} = 0$. We sketch the proof. 

\begin{proposition}\label{prop:unweighted-2-tensors}
Suppose $S$ is a section of $\Lambda^-_g \otimes \Lambda^+_g$ and $h = u \, g +  \mathrm{tr}(S)$ on $(M^4, g)$. Then
\begin{align*}
Lh  &= (\Delta u) g + 2u\,\Ric+ \frac{1}{2}g(\mathring{\mathrm{Ric}}, h)g + \mathrm{tr}\Big(\Delta^L_H S + S \circ \big (W^+ +  \frac{\scal}{6}\mathrm{Id}_{\Lambda^+_g}\big)\Big), \\
Lh& = (\Delta u) g \, + 2 u \, \Ric +   \frac{1}{2}g(\mathring{\mathrm{Ric}},h)g + \mathrm{tr}\Big(\Delta^R_H S + \big (W^- +  \frac{\scal}{6}\mathrm{Id}_{\Lambda^-_g}\big) \circ S \Big).
\end{align*}
\end{proposition} 
\begin{proof}[Proof sketch.]
    We outline the key steps. 
    \begin{enumerate}
        \item Using that $\#$ and $\ast$ commute, and noting that both $\mathrm{Id}_{\Lambda^2}$ and $W$ lies in the subbundle $(\Lambda^+_g \otimes \Lambda^+_g) \oplus \Lambda^-_g\otimes \Lambda^-_g$, one verifies that 
        \[
        \mathrm{Id}_{\Lambda^2}\# S = W\# S = 0.
        \]
        \item Next, from the Weitzenb\"ock formula in Proposition \ref{prop:weitzenbock-bi-forms}, the fact that $\Delta\mathrm{tr}(S) = \mathrm{tr}(\Delta S)$, and (1), one obtains 
        \begin{align*}
        \Delta \mathrm{tr}(S) = \mathrm{tr}\Big(\Delta_H^L S  - \big(W - \frac{\scal}{3} \mathrm{Id}_{\Lambda^2}\big) \circ S - \frac{1}{2}(\mathring{\mathrm{Ric}} \owedge g) \# S\Big), \\
        \Delta \mathrm{tr}(S) = \mathrm{tr}\Big(\Delta_H^R S -  S \circ \big(W - \frac{\scal}{3} \mathrm{Id}_{\Lambda^2}\big)  - \frac{1}{2}S \# (\mathring{\mathrm{Ric}} \owedge g)\Big).
        \end{align*}
        \item Next, one does some preliminary computations towards a formula for $R(\mathrm{tr}(S))$. To that end, recall that curvature operators act on 2-tensors by $T(h)_{ik} = T_{ijkl}h_{jl}$. A computation in a local frame gives the formulas 
        \begin{align*}
        \mathrm{Id}_{\Lambda^2}(\mathrm{tr}(S)) &= -\frac{1}{2}\mathrm{tr}(S), \\
        (\mathring{\mathrm{Ric}}\owedge g)(\mathrm{tr}(S)) &= \frac{1}{2} g\big(\mathring{\mathrm{Ric}}, \mathrm{tr}(S)\big)g + \frac{1}{2}  \mathrm{tr}\big((\mathring{\mathrm{Ric}}\owedge g) \# S\big), \\
        W(\mathrm{tr}(S)) &= \frac{1}{2} \mathrm{tr}(W \circ S + S \circ W).
        \end{align*}
        Note these identities are pointwise and linear, and thus most readily verified via a local frame and a good basis of traceless 2-tensors. For the last formula, one uses the duality properties of the Weyl tensor, along with the fact that it is traceless.  
        \item Putting the identities in the previous step together and using the Ricci decomposition formula \eqref{eq:ricci-decomp}, one obtains 
        \begin{align*}
        2R(\mathrm{tr}(S)) &= \frac{1}{2}g\big(\mathring{\mathrm{Ric}}, \mathrm{tr}(S)\big)g + \mathrm{tr}\Big(-\frac{\scal}{6} S + \frac{1}{2} (\mathring{\mathrm{Ric}} \owedge g) \# S +  W \circ S + S \circ W\Big).
        \end{align*}
        \item For the conformal component, one has $L(ug) = (\Delta u) g + 2 u \, \Ric$. 
        \item Putting steps (2), (4), and (5) together yields the asserted formulas, after some cancellation and recalling that $\#$ is  symmetric. Note additionally that as $S$ is a section of $\Lambda^-_g \otimes \Lambda^+_g$, we have $S \circ W = S \circ W^+$ and $S \circ \mathrm{Id}_{\Lambda^2} = S \circ \mathrm{Id}_{\Lambda^+_g}$ (and similarly for $S$ acting on the right). 
    \end{enumerate}
\end{proof}

\subsection{The weighted formula for 2-tensors} 

Finally, we introduce a weight $f$ and derive the analogue of Proposition \ref{prop:unweighted-2-tensors}. We use the Weitzenb\"ock formula to simplify the traceless part of the deformation, leaving the expression for the conformal component unchanged. In the K\"ahler setting, where the K\"ahler form $\omega$ satisfies $\omega \circ \omega \propto g$ and $\nabla \omega = 0$, one may derive a formula for conformal deformations much like the one for traceless deformations. In general, the tensors $T \in \{\mathrm{Id}_{\Lambda^2} , \mathrm{Id}_{\Lambda^\pm_g}\}$ in $\Lambda^2\otimes \Lambda^2$ are parallel and have traces proportional to the metric. However, applying the Weitzenb\"ock formula to $T$ does not appear to yield a useful expression in the weighted setting.

Recall that $\Ric_f = \Ric + \Hess_f$.

\begin{corollary}\label{cor:weighted-weitzenbock-on-traceless}
Suppose $S$ is a section of $\Lambda^-_g \otimes \Lambda^+_g$ and $h = u g+ \mathrm{tr}(S)$ on $(M^4, g)$. Then for any $a \in [0, 1]$, 
\begin{align*}
L_f h &= (\Delta_f u) g + 2u \,\Ric  + \mathrm{tr}\Big(\Delta^L_{H, f, a} S + S \circ \big(W^+ + \frac{\scal+ 3 \Delta f}{6} \mathrm{Id}_{\Lambda^+_g}\big) \Big)\\
\nonumber & \qquad +  \frac{1}{2}g(\mathring{\mathrm{Ric}}_f,\mathring h)g- a \big(\mathrm{Hess} f \owedge g\big)\circ S  -a(1-a) |\nabla f|^2\, \mathring h, \\
L_f h & =(\Delta_f u) g + 2u \,\Ric + \mathrm{tr}\Big(\Delta^R_{H, f, a} S + \big(W^- + \frac{\scal + 3 \Delta f}{6} \mathrm{Id}_{\Lambda^-_g}\big) \circ S \Big)\\
\nonumber & \qquad+ \frac{1}{2}g(\mathring{\mathrm{Ric}}_f,\mathring h)g- a \; S \circ \big(\mathrm{Hess} f \owedge g\big)   -a(1-a) |\nabla f|^2\, \mathring h.
\end{align*}
\end{corollary}

\begin{proof}[Proof sketch]
We outline the key steps.
\begin{enumerate}
\item With $\mathring h =\mathrm{tr}(S)$, using that $\Delta_f = \Delta - \nabla_{\nabla f}$, the Weitzenb\"ock formula from Proposition \ref{prop:weitzenbock-bi-forms}, and Lemma \ref{lem:weighted-with-a}, one finds:
\begin{align*}
    L_f \mathring h &= L \mathring h - \nabla_{\nabla f} \mathring h \\
     & =   \mathrm{tr}\Big(\Delta^L_{H, f, a} S  +  S \circ \big (W +  \frac{\scal}{6}\mathrm{Id}_{\Lambda^2}\big)\Big)\\
     & \qquad + \frac{1}{2}g(\mathring{\mathrm{Ric}}, \mathring h)g  - (a-\frac{1}{2}) \mathrm{tr}\Big(\big(\mathrm{Hess} f \owedge g\big) \circ S \Big)+ a \Big(\Delta f -(1-a) |\nabla f|^2\Big) \mathring{h}.
\end{align*}
\item Now using that 
\[
\Hess f \owedge g = \frac{\Delta f}{4} g \owedge g + \mathring{\Hess} f \owedge g
\]
and that $S= \frac{1}{2} \pi_{-, +} (\mathring{h} \owedge g)$ 
we observe that
\begin{align*}
    \mathrm{tr}\big((\mathring{\Hess}f \owedge g) \circ S \big) =\mathrm{tr} \big(\pi_{+,-}(\mathring{\Hess}f \owedge g) \circ S \big)  = \frac{1}{2} \mathrm{tr}\big((\mathring{\Hess}f \owedge g)\circ (\mathring{h} \owedge g)\big)
\end{align*}
To further simplify this, we use \eqref{eq:kn-formula} to derive general formulas for traceless 2-tensors $\mathring A, \mathring B$:
\begin{align*}
   (\mathring{A} \owedge g) \circ (\mathring{B} \owedge g) &= 2\mathring{A} \owedge \mathring{B}+(\mathring{A}\mathring{B} + \mathring{B}\mathring{A})  \owedge g, \\
   \mathrm{tr}\big(2\mathring{A} \owedge \mathring{B} + (\mathring{A}\mathring{B} +\mathring{B} \mathring{A} ) \owedge g\big) &= (n-4)(\mathring{A}\mathring{B} +\mathring{B} \mathring{A}) + 2g(\mathring{A}, \mathring{B}) g.
\end{align*}
Hence, with $n = 4$
\begin{align*}
 \mathrm{tr}\big((\Hess f \owedge g) \circ S \big) = (\Delta f) \,\mathrm{tr}(S) +g(\mathring{\Hess} f , \mathring{h}) g 
\end{align*}
\item 
We conclude that
\begin{align*}
     L_f \mathring h &=  \mathrm{tr}\Big(\Delta^L_{H, f, a} S  +  S \circ \big (W +  \frac{\scal}{6}\mathrm{Id}_{\Lambda^2}\big)\Big)\\
     & \qquad +  \frac{1}{2}g(\mathring{\mathrm{Ric}}_f,\mathring h)g- a \big(\mathrm{Hess} f \owedge g\big)\circ S+   \Big(\frac{1}{2}\Delta f -a(1-a) |\nabla f|^2\Big)\mathring h.
\end{align*}
Note that $\frac{1}{2} (\Delta f) \mathring{h} = \mathrm{tr}\big(S \circ \big(\frac{3\Delta f}{6}  \mathrm{Id}_{\Lambda^2}\big) \big)$.
\item For the conformal component, one has $L_f(ug) = (\Delta_f u) g + 2 u \, \Ric$. 
\item Putting together items (3) and (4) yields the left-acting formula. Analogously, one obtains the right-acting formula. 
\end{enumerate}
\end{proof}

\begin{remark}\label{rem: overline R}
    On traceless deformations, the zeroth order term in these formulas is
    \[
    \overline{\mathbf{R}}{}^{\pm}_a := W^{\pm} + \frac{\scal}{6} \mathrm{Id}_{\Lambda^{\pm}_g } + \Big(\frac{\Delta f}{2}  - a(1-a)|\nabla f|^2\Big) \mathrm{Id}_{\Lambda^{\pm}_g} - a \, \Hess f \owedge g.
    \]
    When $f = 0$, this is the self-dual $\mathbf{R}^+$ and anti-selfdual curvature $\mathbf{R}^-$ respectively. When $f \neq 0$, but $a = 0$, we view $\overline{\mathbf{R}}{}^{\pm}=\overline{\mathbf{R}}{}^{\pm}_0$ as a weighted version of the duality curvatures. These weighted curvatures are relevant to the stability of orbifold singularities and appear in \cite{doorb}.
\end{remark}

\subsection{The weighted formula for solitons and K\"ahler solitons}
Suppose now that we take $a = 0$ and consider a gradient Ricci soliton $(M^4, g, f, \lambda)$ satisfying $\Ric_f = \frac{\lambda}{2} g$. Then $\frac{1}{6}(\scal + 3 \Delta f) =\lambda - \frac{1}{3} \scal$. We thus obtain the following corollary. 

\begin{corollary}
Suppose $(M^4, g, f, \lambda)$ is a gradient Ricci soliton. Suppose $S$ is a section of $\Lambda^-_g \otimes \Lambda^+_g$ and $h = u g + \mathrm{tr}(S)$. Then
\begin{align*}
     L_f  h &= (\Delta_f u) g + 2u \, \Ric +   \mathrm{tr}\Big(\Delta^L_{H, f, 0} S  +  S \circ \big (W^+ + (\lambda -  \frac{\scal}{3})\mathrm{Id}_{\Lambda^+_g}\big)\Big), \\
     L_f  h &= (\Delta_f u) g + 2u \, \Ric +   \mathrm{tr}\Big(\Delta^L_{H, f, 0} S  +  \big (W^- + (\lambda -  \frac{\scal}{3})\mathrm{Id}_{\Lambda^-_g}\big) \circ S\Big).
\end{align*}
\end{corollary}

\begin{remark}
    The condition $W^+ - \frac{\scal}{3} \mathrm{Id}_{\Lambda^+_g} > 0$ is known as the half-positive isotropic curvature (half-PIC) condition. That is, in the usual orientation. In the opposite orientation, $W^- - \frac{\scal}{3} \mathrm{Id}_{\Lambda^-_g} > 0$ is also called half-PIC. The combination of both condition is called positive isotropic curvature (PIC). This was introduced to Ricci flow by Hamilton \cite{ham97}, building on an earlier idea of understanding Ricci flow in dimension four through curvature evolution equations \cite{ham86}. 
    
    Conjecturally, the only (simply-connected) half-PIC gradient shrinking Ricci solitons are $\mathbb{CP}^2$, $\mathbb{S}^4$, and $\mathbb{S}^3 \times \mathbb{R}$. It is known that the only (simple-connected) PIC gradient shrinking solitons are $\mathbb{S}^4$ and $\mathbb{S}^3 \times \mathbb{R}$ \cite{lnw18}. 
\end{remark}

\begin{remark}\label{rem:kahler-soliton-conformal}
    When $(M^4, g, f, \lambda)$ is K\"aher with K\"ahler form $\omega$, we have $\omega \circ \omega \propto g$. In this case, the action of $L_f$ on conformal deformations simplifies as follows. Assume we have the normalization $\omega \circ \omega = \frac{1}{4} g$, which implies $|\omega| = 1$.  Using Proposition \ref{prop:weitzenbock-classic} and Lemma \ref{lem:weighted-with-a} with $a = 0$,  we have 
    \begin{align*}
        \frac{1}{4}L_f (u g) & = \Delta_f (u \omega)\circ \omega + \frac{1}{2}u\, \Ric  \\
        & = \big(\Delta_H(u \omega) - \nabla_{\nabla f} (u \omega) \big) \circ \omega  - u(W - \frac{\scal}{3} \mathrm{Id}_{\Lambda^2})(\omega) \circ \omega+ \frac{1}{2}u \, \Ric  \\
        & = \Delta_{H, f, 0} (u \omega) \circ \omega +\frac{1}{2} u( \Hess f\owedge g)(\omega) \circ \omega + \frac{1}{2}u \, \Ric.
    \end{align*}
    We have used that $(W - \frac{\scal}{3} \mathrm{Id}_{\Lambda^2})(\omega) = 0$. 
    Recalling that $\frac{1}{4} g \owedge g = \mathrm{Id}_{\Lambda^2}$, observe   
    \begin{align*}
    \frac{1}{2} u( \Hess f\owedge g)(\omega) \circ \omega + \frac{1}{2}u \, \Ric &= \frac{1}{2} u( \mathring{\Hess f}\owedge g)(\omega) \circ \omega + \frac{1}{2}u \, \mathring{\Ric} + \frac{1}{2} u (\scal + \Delta f)\, \omega \circ \omega. 
    \end{align*}
    The trace of the soliton equation gives $\scal + \Delta f = 2 \lambda$. On the other hand, in the K\"ahler setting $\mathring \Ric = \mathring{\rho} \circ \omega$, where $\mathring{\rho} = \rho - \scal \, \omega$ and $\rho$ is the Ricci form. Then, as $\frac{1}{2}\mathrm{tr}(\mathring{\Ric} \owedge g) = \mathring{\Ric}$, one has
    \[
    \frac{1}{2} \mathring{\Ric} \owedge g = \frac{1}{2}\big(\mathring{\rho} \otimes \omega + \omega \otimes \mathring{\rho}). 
    \]
    Thus,  
    \[
    \frac{1}{2} ( \mathring{\Hess f}\owedge g)(\omega) \circ \omega + \frac{1}{2} \, \mathring{\Ric}   = \frac{1}{2} ( \mathring{\Hess f}\owedge g)(\omega) \circ \omega  + \frac{1}{2} \mathring{\rho} \circ \omega =  \frac{1}{2} (\mathring{\Ric}_f \owedge g)(\omega) \circ \omega  = 0. 
    \]
    Hence
    \begin{equation}\label{eq:kahler-soliton-conformal}
    \frac{1}{4}L_f (u g) 
    = (\Delta_{H, f, 0} +\lambda)(u \omega) \circ \omega.
    \end{equation}
\end{remark}

\end{document}